\documentclass[oneside]{amsart}
\usepackage{amsmath,amssymb,amsfonts,amsthm}
\usepackage{color}
\usepackage{graphicx}
\usepackage{geometry} 
\usepackage{amscd}
\usepackage{enumitem}

\title{O\MakeLowercase{n the classification of } L\MakeLowercase{andsberg   spherically symmetric} F\MakeLowercase{insler metrics}  }

\author[Elgendi]{S. G.~Elgendi}

\address{S. G.~Elgendi, Department of Mathematics, Faculty of Science, Benha
  University, Egypt, \phantom{mmmmm}
 Institute of Mathematics, University of Debrecen,
  Debrecen, Hungary} \email{salah.ali@fsci.bu.edu.eg, \ \ salahelgendi@yahoo.com.}



\def\blue#1{\textcolor[rgb]{0.0,0.0,1.0}{#1}}

\newcommand{\T}{{\mathcal T}}
\newcommand{\C}{{\mathcal C}}

\newcommand{\Real}{\mathbb R}

\newcommand{\To}{\longrightarrow}

\newcommand{\tm}{\T M}

\newcommand{\TM}{\mathcal T\hspace{-1pt}M}

\def\+{\!+\!}

\def\={\!=\!}
\def\<{\!<\!}
\def\>{\!>\!}

\newcommand\undersym[2]{\raisebox{-7pt}{\tiny$#2$}{\kern-8pt}\mbox{$#1$}}
\newcommand\undersymm[2]{\raisebox{-7pt}{\tiny$#2$}{\kern-15pt}\mbox{$#1$}}

 \let\oldmarginpar\marginpar
\renewcommand\marginpar[1]{\oldmarginpar[\raggedleft\footnotesize #1]%
  {\blue{\raggedright \footnotesize \fbox{
      \begin{minipage}{1.0\linewidth}
        #1
      \end{minipage}
}}}}

\numberwithin{equation}{section} 
\numberwithin{figure}{section} 

\theoremstyle{plain}
\newtheorem*{theorem*}{Theorem}
\newtheorem{theorem}{Theorem}[section]

\newtheorem{proposition}[theorem]{Proposition}

\theoremstyle{definition}
\newtheorem{definition}[theorem]{Definition}

\theoremstyle{remark}
\newtheorem{example}{Example}
\newtheorem{remark}[theorem]{Remark}
\newtheorem*{acknowledgement*}{Acknowledgement}

\begin{document}

\maketitle

\begin{abstract}
In this paper,  as an application of the inverse problem of calculus of variations, we investigate  two compatibility conditions on the spherically symmetric Finsler metrics. By making use of these  conditions, we focus our attention on  the Landsberg spherically symmetric Finsler metrics. We classify  all spherically symmetric manifolds of Landsberg or Berwald types. For  the higher dimensions $n\geq 3$, we prove that: all Landsberg spherically symmetric manifolds are  either Riemannian or their geodesic sprays have a specific formula; all regular Landsberg spherically symmetric  metrics are Riemannian; all (regular or non-regular) Berwald  spherically symmetric metrics are Riemannian. Moreover, we establish new unicorns, i.e., new explicit  examples of  non-regular non-Berwaldian  Landsberg metrics  are obtained. For the two-dimensional case, we characterize all Berwald or  Landsberg spherically symmetric  surfaces.   
 \end{abstract}

\section{Introduction}

     A spray $S$ on an $n$-dimensional  manifold $M$ is a vector field on the tangent bundle $TM$ and arises   from  a system of second-order   ordinary
differential equations and   conversely, a spray can be associated with a system of second-order    ordinary differential equations. 
The inverse problem for sprays (or, the so-called Finsler metrizability problem) is the problem of deciding,  for a given spray $S$, whether or not there exists  a Lagrangian or a Finsler function such that $S$ is its  geodesic spray.
The inverse problem for sprays  is   a special case of the inverse problem of the calculus of variations for arbitrary systems of second-order ordinary differential equations. 
Therefore,   a   Finsler function $F$ has  a unique geodesic spray but for a given spray, to decide if it is coming from a Finsler function, one has to solve the inverse problem. 

\medskip

 In a  Berwald manifold  $(M,F)$ and for any piecewise smooth curve $c(t)$ connecting two points $p,q\in M$, the Berwald parallel translation $P_c$ is linear isometry  between $(T_pM,F_p)$  and $(T_qM,F_q)$ and this is equivalent to that the geodesic spray of $F$ is quadratic.  But, in a Landsberg manifold the   parallel translation $P_c$ along $c$ preserves the induced Riemannian metrics on the slit tangent spaces, that is,  $P_c : (T_pM \backslash \{0\},{g}_p) \longrightarrow (T_qM \backslash \{0\}, {g}_q)$ is an isometry.
   This is equivalent to that the horizontal covariant derivative of the metric tensor of $F$ with respect to the Berwald connection vanishes.
It is clear  that every Berwald manifold  is  Landsbergian. Whether there exists a regular Landsberg manifold,   which is not  Berwaldian,  is a long-standing question in Finsler geometry, which is still open.  To the best of our knowledge, it is not known a concrete  example of a regular non-Berwaldian Landsberg manifold.  In \cite{Asanov}, G. S. Asanov  introduced a class of  examples, arising from   general relativity, of non-Berwaldian Landsberg manifolds, of dimension at least $3$.     In \cite{Shen_example}, Z. Shen  generalized Asanov's class to  the  $(\alpha,\beta)$-metrics of Landsberg type. The difficulty  of finding a regular non-Berwaldian Landsberg space leads   Bao \cite{Bao} to describe them as the unicorns of Finsler geometry. For more details and more examples of non-regular non-Berwaldian Landsberg metrics, we refer to, \cite{ Elgendi-LBp,Elgendi-solutions,Elgendi-ST_condition}.
In dimension two, Zhou  \cite{Zhou} provides non-regular examples of non-Berwaldian Landsberg spherically symmetric surfaces. But it was shown in   \cite{Elgendi-Youssef} that  these surfaces   are  Berwaldian. At the end of this paper we discuss this class in some details.

\medskip

Let $\Omega $ be a convex domain in $\mathbb{R}^n$. A Finsler metric $F$ on $\Omega$  is said to be spherically symmetric if the orthogonal group acts as isometries of $F$, that is, $(\Omega,F)$ is  invariant under all rotations in $\mathbb{R}^n$. In other words, $F$ is spherically symmetric if it is orthogonally invariant in the sense that $F(Ax,Ay)=F(x,y)$, where $A\in O(n)$, for more details see \cite{Guo-Mo}. Spherically symmetric metrics were   introduced by Rutz    \cite{Rutz} where she  generalized the classic Birkhoff theorem in general relativity to the Finslerian case. Recently, spherically symmetric Finsler metrics have been paid  a lot of attention, for example, we refer to \cite{Zhou_Mo,Zhou}.

\medskip

Let $|\cdot|$ and $\langle \cdot , \cdot \rangle$ be the standard Euclidean norm and the inner product on $\mathbb{R}^n$.  A Finsler metric $F$ on $\mathbb{B}^n(r_0)\subset\mathbb{R}^n$   is  spherically symmetric if and only if there exists  a positive $C^\infty$ function $\phi $ such that
$F(x,y)= u \phi(r,s)$
where  $r=|x|$, $u=|y|$ and $s=\frac{\langle x, y \rangle}{|y|}$ and $(x,y)\in T\mathbb{B}^n(r_0)\backslash \{0\}$. The geodesic spray $S$ of $F$ is given by
 $$ G^i=uPy^i+u^2 Qx^i.$$

Now,  since a spherically symmetric metric $F$  is written in the form  $F=u\phi$ then we can talk about $\phi$ rather than $F$. For a given function    $\phi$ we have the geodesic spray and hence the functions $P$ and $Q$. But if     $P$ and $Q$ are given, then we have to solve the inverse problem to decide if there exists a function $\phi$ such that the geodesic spray is given by $P$ and $Q$. 

\medskip

In this paper, we solve the inverse problem for the spherically symmetric metrics and find compatibility conditions on  $P$, $Q$ and $\phi$, that is for   given $P$ and $Q$ then the compatibility conditions  provide the function $\phi$ where the geodesic spray is given by $P$ and $Q$ (see Proposition \ref{Comp_C_C_2}).

Depending on the compatibility conditions, we classify all  Landsberg spherically symmetric Finsler metrics; that is, we prove that  a Landsberg spherically symmetric Finsler metric of dimension $n\geq3$ is either Riemannian or the geodesic  spray is given by the functions
$$ P=c_1 s+\frac{c_2}{r^2}\sqrt{r^2-s^2}  , \quad Q=\frac{1}{2}c_0 s^2-\frac{c_2 s}{r^4}\sqrt{r^2-s^2}+c_3,$$
where $c_0$, $c_1$, $c_2$, $c_3$ are arbitrary functions of $r$ (see Theorem \ref{Theorem_A}).   We find out  new non-regular Landsberg spherically symmetric metrics which are not Berwaldian. These metrics are given by $F=u\phi(r,s)$, where $\phi$ is given by 
\small{\begin{equation*}
\begin{split}
\phi(r,s)=&(5s^2-r^2)^{\frac{1}{3}}(2s^2-r^2)^{\frac{1}{6}}\exp\Bigg{(} \frac{1}{3}\operatorname{arctanh} \Bigg{(} \frac{\sqrt{5}}{10 }\frac{(\sqrt{5}s+5r)}{\sqrt{r^2-s^2}}\Bigg{)}-\frac{1}{6}\operatorname{arctanh}\Bigg{(}  \frac{\sqrt{2}}{2  }\frac{(\sqrt{2}s+2r)}{\sqrt{r^2-s^2}}\Bigg{)}\\ 
&
+ \frac{1}{3}\operatorname{arctanh} \Bigg{(} \frac{\sqrt{5}}{2  }\frac{(\sqrt{5}s-5r)}{\sqrt{r^2-s^2}}\Bigg{)}-\frac{1}{6}\operatorname{arctanh}\Bigg{(}  \frac{\sqrt{2}}{  2}\frac{(\sqrt{2}s-2r)}{\sqrt{r^2-s^2}}\Bigg{)}\Bigg{)}
,
\end{split}
\end{equation*}}
\small{\begin{equation*} 
\phi(r,s)=\frac{1}{r} (r^4-5r^2s^2+5s^4)^{\frac{1}{4}} \Big{(} \frac{A}{B}\Big{)}^{\frac{5-\sqrt{5}}{20}}\Big{(} \frac{C}{D}\Big{)}^{\frac{5+\sqrt{5}}{20}}\exp\Bigg{(} -\frac{\sqrt{5}}{10}\operatorname{arctanh} \Bigg{(}  \frac{\sqrt{5}(r^2-2s^2)}{r^2}\Bigg{)}\Bigg{)},
\end{equation*}}
where 
\begin{align*}
A:&=\left(2\sqrt{r^2-s^2}-  2r+\left(\sqrt{5}   -1\right)s\right)\sqrt{r^2-s^2}-\left(\sqrt{5} -1 \right)rs,\\
B:&= \left(2\sqrt{r^2-s^2}-2r-\left(\sqrt{5}   -1\right)s\right)\sqrt{r^2-s^2}+\left(\sqrt{5} -1  \right)rs,\\
C:&=\left(2\sqrt{r^2-s^2}-2r-\left(\sqrt{5}   +1\right)s\right)\sqrt{r^2-s^2}+\left(\sqrt{5} +1  \right)rs,\\
D:&=\left(2\sqrt{r^2-s^2}-2r+\left(\sqrt{5}   +1\right)s\right)\sqrt{r^2-s^2}-\left(\sqrt{5} +1  \right)rs.\\
\end{align*}
For more details about the above two metrics, see Examples \ref{Ex:1} and \ref{Ex:2}.
  We show that    all (regular or non-regular) Berwald spherically symmetric metrics of dimension $n\geq 3$ are Riemannian  (see Theorem \ref{Th_Riemannian}). Also,  we prove that all regular Landsberg spherically symmetric metrics of dimension $n\geq 3$ are Riemannian  (see Theorem  \ref{Reg_Lands_Riem_D_3}).   
 
 \medskip
  
Finally, in Section 6,  we focus our attention on  the two-dimensional case. We characterize all Berwaldian and Landsbergian spherically symmetric surfaces. We prove that a spherically symmetric Finsler surface is Berwald if and only if 
$$P=b_1 s+ \frac{b_2}{ \sqrt{r^2-s^2}} +\frac{b_3 (r^2-2s^2)}{\sqrt{r^2-s^2}} ,$$
$$ Q= b_0 s^2+\frac{1}{2} b_1+ \frac{b_2 s (r^2-2 s^2)}{r^4 \sqrt{r^2-s^2}} -\frac{b_3 s (3r^2-2s^2)}{r^2\sqrt{r^2-s^2}}-\frac{a  }{r^2} s \sqrt{r^2-s^2},$$
where $a$, $b_0$, $b_1$, $b_2$, $b_3$ are arbitrary functions of $r$ and to be chosen such  that the compatibility conditions are satisfied. We discuss the Landsberg surfaces. Moreover, we give a note on a class of Landsberg surfaces mentioned in \cite{Zhou} which,   in fact, is Berwaldian.

\section{Preliminaries}

Let $M$ be an $n$-dimensional manifold and $(TM,\pi,M)$ be its tangent
bundle. We denote by $\TM:=TM\setminus \{0\}$   the set of nonzero tangent
vectors.   Also, $(x^i) $ denote  the local coordinates on the base manifold
$M$ and   $(x^i, y^i)$ are  the induced coordinates on $TM$.   
 The vector $1$-form $J$ on $TM$ defined by $J = \frac{\partial}{\partial y^i} \otimes dx^i$ is called the
natural almost-tangent structure of $T M$. The vertical vector field
$\C=y^i\frac{\partial}{\partial y^i}$ on $TM$ is called the canonical or the
Liouville vector field.

A vector field $S\in \mathfrak{X}(\T M)$ is called a spray if $JS = \C$ and
$[\C, S] = S$. Locally, a spray can be expressed as follows
\begin{equation*}
  \label{eq:spray}
  S = y^i \frac{\partial}{\partial x^i} - 2G^i\frac{\partial}{\partial y^i},
\end{equation*}
where the spray coefficients $G^i=G^i(x,y)$ are $2$-homogeneous
functions in  $y$. 

A nonlinear connection on $TM$ is defined by an $n$-dimensional distribution $H : z \in \tm \rightarrow H_z(\tm) $ which  is supplementary to the vertical distribution. This means that for all $z \in \tm$, we have
$$T_z(\tm) = H_z(\tm) \oplus V_z(\tm).$$

Every spray $S $ induces a canonical nonlinear connection through the corresponding horizontal and vertical projectors,
\begin{equation*}
  \label{projectors}
    h=\frac{1}{2}  (Id + [J,S]), \,\,\,\,            v=\frac{1}{2}(Id - [J,S])
\end{equation*}
 With respect to the induced nonlinear connection,  we have two associated  projectors $h$ and $v$ which are  expressed as follows
$$h=\frac{\delta}{\delta x^i}\otimes dx^i, \quad\quad v=\frac{\partial}{\partial y^i}\otimes \delta y^i,$$
$$\frac{\delta}{\delta x^i}=\frac{\partial}{\partial x^i}-G^j_i(x,y)\frac{\partial}{\partial y^j},\quad \delta y^i=dy^i+G^i_j(x,y)dx^j, \quad G^j_i(x,y)=\frac{\partial G^j}{\partial y^i}.$$

If $X\in  \mathfrak{X}(M)$, $ i_{X}$ and $\mathcal{L}_X$ denote the interior product by $X$ and the Lie derivative  in the direction of $X$, respectively.   Every vector $\ell$-form $L$ defines two graded derivations $i_L$ and $d_L$ of the exterior algebra of $M$ such that
 $$i_Lf=0, \quad i_Ldf=df\circ L, \quad d_L:=[i_L,d]=i_L\circ d-(-1)^{\ell-1}di_L,$$
 where $f\in C^\infty(M)$ and $df$ is the exterior derivative of $f$. 
  If $X\in \mathfrak{X}(M)$ is a vector field, then $i_{X}$ is simply the
interior product by $X$ and $d_X=\mathcal{L}_X$ the Lie derivative with
respect to $X$.

\medskip

 A Finsler function  on $M$ is defined as follows:

\begin{definition}
A Finsler manifold  of dimension $n$ is a pair $(M,F)$, where $M$ is a  differentiable manifold of dimension $n$ and $F$ is a map, called Finsler function or Finsler metric,   $$F: TM \To \Real ,\vspace{-0.1cm}$$  such that{\em:}
 \begin{description}
    \item[(a)] $F$ is smooth and strictly positive on $\T M$ and $F(x,y)=0$ if and only if $y=0$,
    \item[(b)]$F$ is positively homogeneous of degree $1$ in the direction argument $y${\em:}
    $d_{\mathcal{C}} F=\mathcal{L}_\C F=F$,
    \item[(c)] The metric tensor $g_{ij}= \dot{\partial}_i \dot{\partial}_j E$ has maximal rank on $\T M$, where $E:=\frac{1}{2}F^2$ is the energy function.
 \end{description}
 In this case $(M,F)$ is called a regular Finsler space. If $F$ is not smooth or even not defined in some directions, then we call the Finsler function non-regular. In other words, by a non-regular Finsler metric $F$ we mean that $F$ does not satisfy the  condition (a) or (c) on certain directions  of $TM$.
 \end{definition}

 Since the $2$-form $dd_JE$ is non-degenerate,  the Euler-Lagrange equation
\begin{equation*}
  i_Sdd_JE=-dE
\end{equation*}
uniquely determines a spray $S$ on $TM$.  This spray is called the
\emph{geodesic spray} of the Finsler function $F$.

\begin{definition}
  A spray $S$ on a manifold $M$ is called \emph{Finsler metrizable} if there
  exists a Finsler function $F$ such that the geodesic spray of the Finsler
  manifold $(M,F)$ is $S$.
  \end{definition}

  It is known that a spray $S$ is Finsler metrizable if and only if there exists a non-degenerate  solution $F$    for the system
  \begin{equation}
  \label{metrizable_system}
  d_hF=0, \quad d_\C F=F,
  \end{equation}
  where $h$ is the horizontal projector associated to $S$. 
  
  So,  to solve the inverse problem for  a given spray $S$,  we have to find  a non-degenerate solution $F$ for the system \eqref{metrizable_system}.

\section{Spherically symmetric metrics}

Let $|\cdot|$ and $\langle \cdot , \cdot \rangle$ be the standard Euclidean norm and inner product on $\mathbb{R}^n$.  A Finsler metric $F$ on $\mathbb{B}^n(r_0)\subset\mathbb{R}^n$   is  spherically symmetric if and only if there exists a function $\phi:[0,r_0)\times\mathbb{R}^n\to \mathbb{R}$ such that
$$F(x,y)=|y|\phi\left(|x|,\frac{\langle x, y \rangle}{|y|}\right)$$
where $(x,y)\in T\mathbb{B}^n(r_0)\backslash \{0\}$. Or simply we write $F=u\ \phi(r,s)$ where $r=|x|$, $u=|y|$ and $s=\frac{\langle x, y \rangle}{|y|}$.
 
 \medskip

The spherically symmetric metrics are a special case of the general $(\alpha,\beta)$-metrics. The general $(\alpha,\beta)$-metrics are introduced and studied by C. Yu and H. Zhu \cite{Yu-Zhu}. Therefore, the spherically symmetric metrics $F=u\phi(r,s)$ on $\mathbb{B}^n(r_0)$ are regular if and only if $\phi$ is positive and $C^\infty$ function and 
\begin{equation}
\label{Regular_condition}
\phi-s\phi_s>0,\quad \phi-s\phi_s+(r^2-s^2)\phi_{ss}>0
\end{equation}
when $n\geq 3$ or $\phi-s\phi_s+(r^2-s^2)\phi_{ss}>0$  when $n=2$ for all $|s|\leq r<r_0$.  
Throughout, the subscript $s$ (resp. $r$) refers to the derivative with respect to $s$ (rep. $r$). 
The spherically symmetric Finsler metrics are studied in many papers for example, we refer   to, \cite{Guo-Mo,Zhou_Mo,Zhou}.

\begin{remark}
It should be noted that throughout, the results and formulae work and are valid for both regular and non-regular spherically symmetric metrics unless it is specified. So we consider the most general case,  that is, the non-regular case and when the situation needs, we mention clearly that the metric is regular.
\end{remark}

The components $g_{ij}$ of the metric tensor of the spherically symmetric metric $F=u \phi(r,s)$ are given by  
\begin{align}
\label{Eq:g^ij}
g_{ij}=&\sigma_0\ \delta_{ij}+\sigma_1\ \delta_{ik}\delta_{hj}x^kx^h+\frac{\sigma_2}{u} (\delta_{ik}\delta_{hj}+\delta_{jk}\delta_{hi})x^hy^k+\frac{\sigma_3}{u^2}\delta_{ik}\delta_{hj}y^ky^h,
\end{align}
where  $$\sigma_0=\phi(\phi-s\phi_s),\quad \sigma_1= \phi_s^2+\phi\phi_{ss},\quad \sigma_2= (\phi-s\phi_s)\phi_s-s\phi\phi_{ss},  \quad \sigma_3= s^2\phi\phi_{ss}-s(\phi-s\phi_s)\phi_s.$$
 The above formula can be found in \cite{Zhou} and to keep the indices consistent in  both sides of a tensorial  equation,   we  write  for example,   $\delta_{ik}y^k$ when there is a  $y^i$ and $i$ is a lower index.

\medskip

The components $g^{ij}$ of the inverse metric tensor are given by \cite{Zhou} as follows
\begin{align}
\label{Eq:g^ij}
g^{ij}=&\rho_0\delta^{ij}+ \frac{\rho_1}{u^2}y^iy^j+ \frac{\rho_2}{u} (x^iy^j+x^jy^i)+\rho_3x^ix^j,
\end{align}
where
$$
 \rho_0=\frac{1}{\phi(\phi-s\phi_s)}, \quad \rho_1=\frac{(s\phi+(r^2-s^2)\phi_s)(\phi\phi_s-s\phi_s^2-s\phi_{ss})}{\phi^3(\phi-s\phi_s)(\phi-s\phi_s+(r^2-s^2)\phi_{ss})},$$
 $$\rho_2=-\frac{\phi\phi_s-s\phi_s^2-s\phi_{ss}}{\phi^2(\phi-s\phi_s)(\phi-s\phi_s+(r^2-s^2)\phi_{ss})},\quad \rho_3=-\frac{\phi_{ss}}{\phi(\phi-s\phi_s)(\phi-s\phi_s+(r^2-s^2)\phi_{ss})}.$$
  Moreover, the geodesic spray coefficients  $G^i$ of $F$ are  given by 
\begin{equation}\label{G}
  G^i=uPy^i+u^2 Qx^i,
\end{equation}
where the functions $P$ and $Q$ have the following formulae
\begin{equation}\label{P,Q}
 Q:=\frac{1}{2 r}\frac{ -\phi_r+s\phi_{rs}+r\phi_{ss}}{\phi-s\phi_s+(r^2-s^2)\phi_{ss}}, \quad P:=-\frac{Q}{\phi}(s\phi +(r^2-s^2)\phi_s)+\frac{1}{2r\phi}(s\phi_r+r\phi_s).
\end{equation}

For the geodesic spray of the  spherically symmetric metric $F$, we have the associated nonlinear connection whose  components $G^i_j$ are defined by  $G^i_j:=\frac{\partial G^i}{\partial y^j}$. Using the following identities:
\begin{equation}
\label{Eq:derivatives}
    \begin{split}
        \frac{\partial r}{\partial x^j}&= \frac{1}{r}\delta_{jk}x^k, \quad  \frac{\partial u}{\partial y^j}= \frac{1}{u}\delta_{jk}y^k,  \quad  \frac{\partial u}{\partial x^j}= 0,         \quad  \frac{\partial r}{\partial y^j}= 0,   \\
         \frac{\partial s}{\partial x^j}&= \frac{1}{u}\delta_{jk}y^k, \quad  \frac{\partial s}{\partial y^j}= \frac{1}{u}\delta_{jk}x^k-\frac{s}{u^2}\delta_{jk}y^k ,   \\  
    \end{split}
\end{equation}
the components $G^i_j$ can be calculated as follows
\begin{equation}
\label{Eq:G^i_j}
G^i_j=u P \delta^i_j+P_s \delta_{kj}x^ky^i+  \frac{1}{u}(P-sP_s) \ \delta_{kj}y^ky^i+u Q_s\delta_{jk}x^ix^k+(2Q-sQ_s)\delta_{jk}x^iy^k.
\end{equation}
The components $G^i_{jk\ell}$ of the Berwald curvature are defined by
$$G^i_{jk\ell}=\frac{\partial^2 }{\partial y^\ell \partial y^k} G^i_j.$$
For a spherically symmetric Finsler metric $F=u\ \phi(r,s)$,  we use the identities \eqref{Eq:derivatives} and \eqref{Eq:G^i_j} to find the components $G^i_{jk\ell}$ as follows
\begin{equation*}
    \begin{split}
        G^i_{jk\ell}= & \frac{P_{ss}}{u}(\delta^i_j \delta_{hk}\delta_{t\ell}+\delta^i_\ell \delta_{hk}\delta_{tj}+\delta^i_k \delta_{hj}\delta_{t\ell})x^hx^t+\frac{1}{u}(P-sP_s)(\delta^i_j \delta_{k\ell}+\delta^i_k \delta_{j\ell}+\delta^i_\ell \delta_{jk})\\
        &-\frac{sP_{ss}}{u^2}(\delta^i_j(\delta_{hk}\delta_{t\ell}+\delta_{h\ell}\delta_{tk})+\delta^i_k(\delta_{hj}\delta_{t\ell}+\delta_{h\ell}\delta_{tj})+\delta^i_\ell(\delta_{hk}\delta_{tj}+\delta_{hj}\delta_{tk}))x^hy^t\\
        &-\frac{sP_{ss}}{u^2}(\delta_{jk}\delta_{h\ell}+\delta_{j\ell}\delta_{kh}+\delta_{k\ell}\delta_{jh})x^h y^i
        +\frac{1}{u}(Q_s-sQ_{ss})(\delta_{jk}\delta_{h\ell}+\delta_{j\ell}\delta_{kh}+\delta_{k\ell}\delta_{jh})x^i x^h\\
        &+\frac{1}{u^3}(s^2P_{ss}+sP_s-P)((\delta^i_j\delta_{kh}\delta_{t\ell}+\delta^i_k\delta_{jh}\delta_{t\ell}+\delta^i_\ell\delta_{kh}\delta_{tj})y^hy^t+(\delta_{jk}\delta_{\ell h}+\delta_{j\ell}\delta_{k h}+\delta_{k\ell}\delta_{j h})y^hy^i)\\
        &+\frac{1}{u^5}(3P-s^3P_{sss}-6s^2P_{ss}-3sP_s)\delta_{jh}\delta_{km}\delta_{\ell t} y^hy^my^ty^i+\frac{P_{sss}}{u^2}\delta_{jh}\delta_{k m}\delta_{\ell t} x^hx^mx^ty^i\\
        &+\frac{1}{u^4}(s^2P_{sss}+3sP_{ss})(\delta_{jm}\delta_{kh }\delta_{\ell t} +\delta_{jm}\delta_{kt }\delta_{\ell h}+\delta_{\ell m}\delta_{kh }\delta_{j t}  )x^ty^my^hy^i\\
        &-\frac{1}{u^3}(P_{ss}+sP_{sss})(\delta_{km}\delta_{jh }\delta_{\ell t} +\delta_{jm}\delta_{kh }\delta_{\ell t}+\delta_{jm}\delta_{\ell h }\delta_{k t})x^tx^my^hy^i+\frac{Q_{sss}}{u}\delta_{km}\delta_{jh }\delta_{\ell t}x^hx^mx^tx^i\\
        &+\frac{1}{u^3}(s^2Q_{sss}+sQ_{ss}-Q_s)(\delta_{km}\delta_{jh }\delta_{\ell t}+\delta_{\ell m}\delta_{kh }\delta_{j t}+\delta_{\ell h}\delta_{km }\delta_{j t})x^hx^iy^my^t
        -\frac{sQ_{sss}}{u^2}(\delta_{\ell m}\delta_{jh }\delta_{k t}\\
        &+\delta_{\ell t}\delta_{kh }\delta_{j m}+\delta_{\ell m}\delta_{kh }\delta_{j t})x^hx^ix^my^t+\frac{1}{u^4}  (3sQ_s-3s^2Q_{ss}-s^3Q_{sss})\delta_{\ell m}\delta_{kh }\delta_{j t}y^ty^my^hx^i\\
        &+\frac{1}{u^2}(s^2Q_{ss}-sQ_s)(\delta_{k\ell }\delta_{hj }+\delta_{j\ell }\delta_{hk }+\delta_{kj }\delta_{h\ell })y^hx^i.
    \end{split}
\end{equation*}
The above formula of the Berwald curvature can be found, for example,  in \cite{Zhou_Mo}.  By straightforward calculations, the components $E_{ij}:=G^h_{ijh}$ of the mean Berwald curvature can be calculated as follows: 
\begin{equation}
\label{Mean_curv.}
 \begin{split}
E_{ij} =&\frac{1}{u}((n+1)(P-sP_s)+(r^2-s^2)(Q_s-sQ_{ss}))\delta_{ij}+ \frac{1}{u^3}((n+1)(s^2P_{ss}+sP_s-P)\\
&+r^2(s^2Q_{sss}+sQ_{ss}-Q_s)+3s^2Q_s-3s^3Q_{ss}-s^4Q_{sss})\delta_{ih}\delta_{jk} y^h y^k\\
&+\frac{1}{u}((n+1)P_{ss}+2(Q_s-sQ_{ss})+(r^2-s^2)Q_{sss})\delta_{ih}\delta_{jk}x^hx^k\\
&-\frac{s}{u^2}((n+1)P_{ss}+2(Q_s-sQ_{ss})+(r^2-s^2)Q_{sss})(\delta_{ih}\delta_{jk}+\delta_{ik}\delta_{jh}) x^h y^k .
\end{split}
\end{equation}
One can rewrite  $E_{ij}$ as follows 
\begin{equation} 
\label{E_H}
E_{ij}=\frac{H}{u}\delta_{ij}-\frac{1}{u^3}(sH_s+H) \delta_{hi}\delta_{kj}y^h y^k+\frac{H_s}{su^2}(s(\delta_{hi}\delta_{kj}+\delta_{ki}\delta_{hj})x^hy^k-u \delta_{hi}\delta_{kj}x^hx^k),
\end{equation}
where 
 $$H:=(n+1)(P-sP_s)+(r^2-s^2)(Q_s-sQ_{ss}).$$
 Now,   we have the following proposition.
\begin{proposition} 
If the mean Berwald curvature vanishes, then either $n=2$ or $H=0$.
\end{proposition} 
 \begin{proof}
 Assume that $E_{ij}=0$, then contracting \eqref{E_H} by $\delta^{ij}$, and using the facts that $\delta_{ij}y^iy^j=u^2$, $\delta_{ij}x^ix^j=r^2$ and $\delta_{ij}x^iy^j=\langle x,y\rangle=su$, we have
 $$ s(n-1)H-(r^2-s^2)H_s=0.$$
 Contracting \eqref{E_H} by $x^{i}x^{j}$ implies
 $$ sH-(r^2-s^2)H_s=0.$$
 Subtracting the above two equations,    yields
 $$s(n-2)H=0.$$
 Which holds for all $s$, then we obtain that $n=2$ or $H=0$ and hence the proof is completed.
  \end{proof}

  To avoid confusions, we have to clarify a point   mentioned in \cite[Proposition 2.1]{Zhou} by providing the following remark.
  \begin{remark} 
  It is clear that if $H=0$, then the mean Berwald curvature vanishes, that is $E_{ij}=0$. By making use of the above proposition, when $n=2$, the condition $E_{ij}=0$ does not imply  $H=0$ contrary to what mentioned in  \cite{Zhou} and  discussed in \cite{Elgendi-Youssef}.    For the higher dimensions   $n\geq3$,  $E_{ij}=0$ if and only if $H=0$.
  \end{remark}

  The scalar trace $E$ of the mean Berwlad curvature is $E=g^{ij}E_{ij}$.  For a spherically symmetric Finsler metric,   the scalar trace of the mean Berwald curvature is given as follows.
  \begin{proposition}\label{Scalar_E}
  The scalar trace $E$ of the mean Berwald curvature of a spherically symmetric Finsler metric is given by
  $$E=\frac{1}{su}\left(s((n-1)\rho_0+\rho_3(r^2-s^2))H-(r^2-s^2)(\rho_0+\rho_3(r^2-s^2))H_s\right).$$  
  \end{proposition}
  \begin{proof}
 By using \eqref{Eq:g^ij} and the fact that $y^iE_{ij}=0$,  the scalar trace of the mean Berwald curvature is calculated as follows 
   \begin{equation*}
    \begin{split}
       E=&g^{ij}E_{ij}\\
       =&\frac{\rho_0}{s u}\left((n-1)sH -(r^2-s^2)H_s)+\frac{\rho_3}{s u}(s(r^2-s^2)H-(r^2-s^2)^2H_s\right)\\
       =&\frac{1}{su}\left(s((n-1)\rho_0+\rho_3(r^2-s^2))H-(r^2-s^2)(\rho_0+\rho_3(r^2-s^2))H_s\right). 
    \end{split}
\end{equation*}
  \end{proof}

The components $L_{jk\ell}$ of the Landsberg curvature are defined by
$$L_{jk\ell}=-\frac{1}{2} FG^h_{jk\ell} \frac{\partial F}{\partial y^h}.$$
For a spherically symmetric Finsler metric $F=u\ \phi(r,s)$,   the components $L_{jk\ell}$ are given by
\begin{equation}
\label{L_ijk}
    \begin{split}
       L_{jk\ell}=& -\frac{\phi}{2}\Big{(}L_1\delta_{jh}\delta_{km}\delta_{\ell t} x^hx^mx^t+\frac{3sL_2-s^3L_1}{u^3}\delta_{jh}\delta_{km}\delta_{\ell t} y^hy^my^t+L_2(\delta_{k\ell}\delta_{jh}+\delta_{j\ell}\delta_{kh}\\
       &+\delta_{k j}\delta_{\ell h})x^h-\frac{sL_2}{u}(\delta_{k\ell}\delta_{hj}+\delta_{kh}\delta_{\ell j}+\delta_{kj}\delta_{h\ell})y^h  -   \frac{sL_1}{u}(\delta_{m k}\delta_{t\ell}\delta_{hj} +\delta_{kh}\delta_{m j}\delta_{t\ell}\\
       & +\delta_{mj}\delta_{h\ell}\delta_{k t})x^mx^ty^h + \frac{s^2L_1-L_2}{u^2}(\delta_{m k}\delta_{t\ell}\delta_{hj}+\delta_{kh}\delta_{m j}\delta_{t\ell}+\delta_{mj}\delta_{h\ell}\delta_{k t})y^my^tx^h \Big{)}
    \end{split}
\end{equation}
where
$$L_1=  3\phi_sP_{ss}+\phi P_{sss}+(s\phi+(r^2-s^2)\phi_s)Q_{sss},   $$
        $$ L_2=  -s\phi P_{ss}+\phi_s(P-sP_{s})+(s\phi+(r^2-s^2)\phi_s)(Q_{s}-sQ_{ss}).$$
    
    It is well known that a Finsler metric $F$ is said to be Berwaldain if and only if its Berwald curvature vanishes, that is, $G^h_{ijk}=0$ and this is equivalent to the fact that the geodesic spray of $F$ is quadratic in $y$. Also, $F$ is called Landsbrgain if and only if its Landsberg  curvature vanishes, that is, $L_{ijk}=0$. Paying more attention to the spherically symmetric Finsler metrics of  Landsberg type, we have the following proposition.
\begin{proposition}\label{Landsberg_cond}
If a spherically symmetric Finsler metric $F=u\phi$ is Landsbergian then  either $n=2$ or  $L_1=L_2=0$.
\end{proposition} 
 \begin{proof}
 Let $F$ be a Landsberg spherically symmetric Finsler metric, then $L_{jk\ell}=0$.  Contracting \eqref{L_ijk} by $x^j\delta^{k\ell}$ and using the facts that $\delta_{ij}y^iy^j=u^2$, $\delta_{ij}x^ix^j=r^2$ and $\delta_{ij}x^iy^j=\langle x,y\rangle=su$, we have
$$(r^2-s^2) L_1+(n+1)L_2=0.$$
Similarly, contracting \eqref{L_ijk} by $x^jx^kx^\ell$, we get
$$(r^2-s^2) L_1+3L_2=0.$$ 
Subtracting the above two equations implies
$$(n-2)L_2=0.$$
Therefore,  $n=2$ or  $L_2=0$ and this completes the proof.
  \end{proof}
  By the above proposition together with the fact that if  $L_1=L_2=0$ then $L_{jk\ell}=0$, one can see that   a Landsbergian spherically symmetric Finsler metric of dimension $n\geq 3$ is characterized by the following two conditions
\begin{equation}
\label{Zhou_L_1_2}
    \begin{split}
        L_1= & 3\phi_sP_{ss}+\phi P_{sss}+(s\phi+(r^2-s^2)\phi_s)Q_{sss}=0,   \\
         L_2=& -s\phi P_{ss}+\phi_s(P-sP_{s})+(s\phi+(r^2-s^2)\phi_s)(Q_{s}-sQ_{ss})=0.
    \end{split}
\end{equation}
  
  Again, to avoid confusions in the class of    the spherically symmetric Finsler metrics of Landsberg type,  we have  the following remark.
  \begin{remark}
  It is clear that if $L_1=L_2=0$ then $L_{jk\ell}=0$ and hence the metric $F$ is Landsbergian. By the above proposition, when $n=2$,  $L_{jk\ell}=0$  does not yield $L_1=L_2=0$ as be discussed in Section 6. For the higher dimensions $n\geq 3$,   $L_{jk\ell}=0$ if and only if $L_1=L_2=0$.  
  \end{remark}

\section{The compatibility conditions}

For a non-regular spherically symmetric metric $F=u\phi(r,s)$,  the case  $\phi-s\phi_s=0$ is excluded. If    $\phi-s\phi_s=0$, then  $\phi=f(r) s$ and hence $F=f(r) us=f(r)\langle x,y\rangle$ which is linear in $y$ for an arbitrary function $f(r)$. If the Finsler function is linear in $y$, then the  metric tensor is degenerate everywhere and this is impossible. Moreover, if $\phi_s=0$, then $F=f(r) u$  which is the Euclidean metric or a conformal transformation of the Euclidean metric and this case also is excluded.  
\begin{proposition}\label{Q_not_defined} 
 Let $F=u\phi(r,s)$ be a spherically symmetric metric, then the geodesic spray of $F$ is not defined  if and only if $\phi-s\phi_s+(r^2-s^2)\phi_{ss}=0$ or equivalently $\phi$ is given by
 $$\phi(r,s)=f_1(r) s+f_2(r) \sqrt{r^2-s^2},$$
 where $f_1(r)$ and $f_2(r)$ are arbitrary functions of $r$.
 \end{proposition}
 \begin{proof} 
 By \eqref{P,Q}, it is clear that the function $Q$ is not defined and so the geodesic spray if and only if 
  $\phi-s\phi_s+(r^2-s^2)\phi_{ss}=0$. Since the above formula holds for all  $r$  and $s$, then we can write  
  $$\frac{\phi_{ss}}{\phi-s\phi_s}=-\frac{1}{r^2-s^2}.$$
  Since $(\phi-s\phi_s)_s=-s\phi_{ss}$,    integrating the above equation with respect to $s$ gives
  $$\phi-s\phi_s=\frac{f_1(r)}{\sqrt{r^2-s^2}}.$$
  Which has the solution
  $$\phi(r,s)=f_1(r) s+f_2(r) \sqrt{r^2-s^2}.$$
 \end{proof}

For a given  spherically symmetric Finsler metric $F=u \phi$, then the geodesic spray \eqref{G} is given by    the functions $P$ and $Q$. But if the functions $P$ and $Q$ are given, then to find the Finsler metric whose    geodesic spray is determined  by $P$ and $Q$,  we have to solve the inverse problem or the metrizability problem. So, we investigate the following two compatibility conditions.
\begin{proposition}
Let $P(r,s)$ and $Q(r,s)$ be given, then the Finsler function $F=u\phi(r,s)$ whose geodesic spray given by $P$ and $Q$ is determined by the function $\phi$ provided that $\phi$ satisfies the following   compatibility conditions:
\begin{equation}
\label{Comp_C_C_2}
    \begin{split}
       C_1:= &(1+sP-(r^2-s^2)(2Q-sQ_s))\phi_s+(s P_s-2P-s(2Q-sQ_s))\phi =0,   \\
        C_2:=& \frac{1}{r}\phi_r-(P+Q_s(r^2-s^2))\phi_s-(P_s+sQ_s) \phi =0.
    \end{split}
\end{equation}
\end{proposition}
\begin{proof}Let $P(r,s)$ and $Q(r,s)$ be two given functions. To find the Finsler function $F=u\phi(r,s)$ whose geodesic spray determined by $P$ and $Q$,  we have to solve the system  $d_hF=0$  \eqref{metrizable_system}.  Locally,  the system  $d_hF=0$  gives rise to
\begin{equation}
\label{Eq:sys_compt}
\frac{\partial (u \phi)}{\partial x^j}-G^i_j\frac{\partial (u \phi)}{\partial y^i}=0.
\end{equation}
We have the following properties
$$\frac{\partial  \phi }{\partial x^i}=\frac{\phi_r}{r}\delta_{ik}x^k+\frac{\phi_s}{u}\delta_{ik}y^k, \quad \frac{\partial  \phi }{\partial y^i}= \frac{\phi_s}{u}\delta_{ik}x^k-\frac{s \phi_s}{u^2}\delta_{ik}y^k . $$
Moreover, since $u\phi$ is homogeneous of degree one in $y$, then $y^j\frac{\partial  (u\phi) }{\partial y^j}=u\phi$. Also, we have $x^j\frac{\partial  (u\phi) }{\partial y^j}=s\phi+(r^2-s^2)\phi_s$. Now,  substituting the above formulae into the system \eqref{Eq:sys_compt}, we have
\begin{align*}
\frac{\partial (u \phi)}{\partial x^j}-G^i_j\frac{\partial (u \phi)}{\partial y^i} =& \frac{u\phi_r}{r}\delta_{jk}x^k+\phi_s\delta_{jk}y^k
 -\Big{(} u P \delta^i_j+P_s \delta_{kj}x^ky^i+  \frac{1}{u}(P-sP_s) \ \delta_{kj}y^ky^i\\
 &+u Q_s\delta_{jk}x^ix^k+(2Q-sQ_s)\delta_{jk}x^iy^k \Big{)}\frac{\partial (u \phi)}{\partial y^i}\\
 =&u\left(\frac{1}{r}\phi_r-P\phi_s-\phi P_s-Q_s(s\phi+(r^2-s^2)\phi_s))\right)\delta_{jk}x^k\\
 &+\left( \phi_s-2P\phi+s\phi P_s+sP\phi_s-(2Q-sQ_s)(s\phi+(r^2-s^2)\phi_s) \right)\delta_{jk}y^k\\
 =&0.
\end{align*}
Now, if there exist two functions $\mu$ and $\eta$  such that
$$\mu \delta_{jk} x^k+\eta \delta_{jk} y^k=0.$$ 
Then by contracting the above equation by $x^j$, resp. by $y^j$ we get the following two equations 
$$r^2 \mu +su \eta=0, \quad s\mu+u\eta=0.$$
Then, we have  $(r^2-s^2) \mu=0$ which holds for all $r$ and $s$, that is, $\mu=0$ and hence $\eta=0$. Using this property, we get the conditions $C_1$ and $C_2$. 

Conversely, assume that the conditions $C_1$ and $C_2$ are satisfied. Rewriting $C_1$ and $C_2$ as follows
\begin{equation}
\label{Rewrite_C_1}
\phi_s+s(P_s \phi+P \phi_s)-2P \phi-(s\phi+(r^2-s^2)\phi_s)(2Q-sQ_s)=0,
\end{equation}
\begin{equation}
\label{Rewrite_C_2}
\frac{1}{r}\phi_r-(P\phi_s+P_s\phi)-(s\phi+(r^2-s^2)\phi_s)Q_s =0.
\end{equation}
Adding \eqref{Rewrite_C_1} to the multiplication of \eqref{Rewrite_C_2} by $s$, we have
$$\phi_s+\frac{s}{r}\phi_r-2P\phi -2Q(s\phi+(r^2-s^2)\phi_s)=0.$$
Which implies 
$$P=\frac{Q}{\phi}(s\phi+(r^2-s^2)\phi_s)+\frac{1}{2r\phi}(r\phi_s+s\phi_r).$$
By the above formula of $P$, we get
$$P\phi_s+\phi P_s=\frac{1}{2r}(\phi_r+s\phi_{rs}+r\phi_{ss}+2rQ(s\phi_s-\phi-(r^2-s^2)\phi_{ss}))-Q_s(s\phi+(r^2-s^2)\phi_s).$$
Now,  substituting  $P\phi_s+\phi P_s$ into \eqref{Rewrite_C_2}, we obtain that
$$Q=\frac{1}{2r} \frac{-\phi_r+s\phi_{rs}+r\phi_{ss}}{\phi-s\phi_s+(r^2-s^2)\phi_{ss}}.$$
This completes the proof.
\end{proof}

\section{ Landsberg spherically symmetric metrics }

In this section we focus our attention  to the spherically symmetric metrics of Landsberg type. We start by the following result.
\begin{theorem}\label{Theorem_A}
A Landsberg spherically symmetric Finsler metric of dimension $n\geq 3$ is either Riemannian or the geodesic  spray is determined  by the functions 
 \begin{equation}
\label{Zhou_P&Q}
P=c_1 s+\frac{c_2}{r^2}\sqrt{r^2-s^2}  , \quad Q=\frac{1}{2}c_0 s^2-\frac{c_2 s}{r^4}\sqrt{r^2-s^2}+c_3,
\end{equation}
where $c_0$, $c_1$, $c_2$, $c_3$ are arbitrary functions of $r$.  
\end{theorem}
\begin{proof}
Suppose  that $F=u\phi(r,s)$ be a spherically symmetric metric of dimension $n\geq 3$, then  by Proposition \ref{Landsberg_cond}, the function $\phi$ and  the functions $P$ and $Q$ of the geodesic spray satisfy the Landsberg conditions  \eqref{Zhou_L_1_2}, that is,  we have $L_1=L_2=0$. The condition $L_2=0$ implies that
$$Q_s-sQ_{ss}=-\frac{-s\phi P_{ss}+\phi_s(P-sP_{s})}{s\phi+(r^2-s^2)\phi_s}.$$
It should be noted  that if $s\phi+(r^2-s^2)\phi_s=0$, then by taking the derivative with respect to $s$ implies $\phi-s\phi_s+(r^2-s^2)\phi_{ss}=0$ which is impossible by Proposition \ref{Q_not_defined}. Now, making use of the property that $(Q_s-sQ_{ss})_s=-sQ_{sss}$, then differentiating the above equation with respect to $s$, we get  $Q_{sss}$. Substituting  $Q_{sss}$ into the Landsberg condition $L_1=0$ and straightforward simplifications, we have
$$((r^2-s^2)P_{ss}+(P-sP_s))(s\phi_s^2+s\phi\phi_{ss}-\phi\phi_{s})=0.$$
If $s\phi_s^2+s\phi\phi_{ss}-\phi\phi_{s}=0$, then we can write
$$\frac{-s\phi_{ss}}{\phi -s\phi_s} =-\frac{\phi_s}{\phi}.$$
By integration with respect to $s$, we have 
$$\phi -s\phi_s=\frac{f_1(r)}{\phi}.$$
Using the substitution $\phi=s \psi$, then we can rewrite the above equation as follows
$$\psi\psi_s=-\frac{f_1(r)}{s^3}.$$
Which implies
$$\phi=\sqrt{f_1(r)+f_2(r)s^2}.$$
That is, the Finsler function  given by $F=u\phi=\sqrt{f_1(r) u^2+f_2(r)\langle x, y\rangle^2}$  is Riemannian.

If $(r^2-s^2)P_{ss}+(P-sP_s)=0$, then by assuming that $P-sP_s\neq 0$ we can write
$$\frac{-sP_{ss}}{P-sP_s}=\frac{s}{r^2-s^2}.$$
Since $(P-sP_s)_s=-sP_{ss}$, then integrating the above equation with respect to $s$ gives rise to
$$P-sP_s=\frac{ c_1(r)}{\sqrt{r^2-s^2}}.$$
Which has the solution 
$$P=c_2(r)s+\frac{c_1(r)}{r^2}\sqrt{r^2-s^2}.$$
Taking the fact that $s\phi+(r^2-s^2)\phi_s\neq 0$, then substituting by the function $P$ into the Landsberg condition $L_2=0$, implies that
$$Q_s-sQ_{ss}=-\frac{c_1(r)}{(r^2-s^2)^{3/2}}.$$
Using the substitution $s\psi=Q_s$, then we can write 
$$s^2\psi_{s}=\frac{c_2(r)}{(r^2-s^2)^{3/2}}.$$
Which yields the solution
$$Q_s=c_0(r) s-\frac{c_2(r) (r^2-2s^2)}{r^4\sqrt{r^2-s^2}}.$$
Which has the solution 
$$Q=\frac{1}{2}c_0 (r) s^2-\frac{c_2 (r)s}{r^4}\sqrt{r^2-s^2}+c_3(r).$$

In the case where   $P-sP_s=0$ then we have $(r^2-s^2)P_{ss}=0$ for all $r$ and $s$ and hence we obtain that $P_{ss}=0$ and also  the Landsberg condition $L_2=0$, implies that $Q_s-sQ_{ss}=0$ and hence $Q_{sss}=0$. Therefore,  $P=f_1(r)s $ and $Q=f_2(r)s^2+f_3(r)$ which are a special case of the given formulae for $P$ and $Q$. Hence the proof is completed.
 \end{proof}
Based on the above proof, we provide the following remark.
\begin{remark}
It should be noted that  if $P-sP_s=0$ and $Q_s-sQ_{ss}=0$, then $P_{ss}=0$ and $Q_{sss}=0$. That is the Berwald curvature vanishes and hence the metric is Berwaldian. Moreover, if $n\geq 3$ and the Berwald curvature vanishes, then we get $P-sP_s=0$ and $Q_s-sQ_{ss}=0$.  That is, as it was mentioned in \cite{Guo-Mo,Zhou_Mo}, the Berwaldain spherically symmetric Finsler metrics are characterized by $P-sP_s=0$ and $Q_s-sQ_{ss}=0$ provided that $n\geq 3$.

\end{remark}

\begin{theorem}\label{main_theorem}
There exist non-Berwaldian Landsberg spherically symmetric Finsler metrics of dimension $n\geq 3$. 
\end{theorem}

\begin{proof}
Let $F=u\phi$ be  a non-Riemannian Landsberg spherically symmetric Finsler metric  of dimension $n\geq 3$, then by Theorem \ref{Theorem_A} the functions  $P$ and $Q$ are given by
\begin{equation*}
P=c_1 s+\frac{c_2}{r^2}\sqrt{r^2-s^2}  , \quad Q=\frac{1}{2}c_0 s^2-\frac{c_2 s}{r^4}\sqrt{r^2-s^2}+c_3.
\end{equation*}
Now,    the functions $P$, $Q$ and $\phi$  must satisfy the compatipility conditons \eqref{Comp_C_C_2}.  Since $\phi$ is non-zero then dividing the conditions $C_1=0$ and $C_2=0$ in  \eqref{Comp_C_C_2} by $\phi$ and substituting by $P$ and $Q$,  straightforward calculations imply  the following two equations:
$$\frac{\phi_s}{\phi}(r^2+(c_1+2c_3)r^2s^2-2c_3r^4+2c_2s\sqrt{r^2-s^2})-(c_1+2c_3)sr^2 -2c_2 \sqrt{r^2-s^2}=0,$$
$$\frac{\phi_s}{\phi}( -c_0r^4s(r^2-s^2)-c_1sr^4-2c_2s^2\sqrt{r^2-s^2})+r^3\frac{\phi_r}{\phi}-c_0s^2r^4-c_1r^4+2c_2s\sqrt{r^2-s^2}=0.$$
Solving the above equations algebraically for $\frac{\phi_s}{\phi}$ and $\frac{\phi_r}{\phi}$, we obtain that
\begin{equation}
\label{Zhou_phi}
    \begin{split}
        \frac{\phi_s}{\phi}= & \frac{(c_1+2c_3)r^2s+2c_2\sqrt{r^2-s^2}}{r^2+(c_1+2c_3)r^2s^2-2c_3r^4+2c_2s\sqrt{r^2-s^2}},   \\
         \frac{\phi_r}{\phi}=&  \frac{(2c_0c_2r^4+4(c_1+c_3)c_2r^2-2c_2)s\sqrt{r^2-s^2}}{r(r^2+(c_1+2c_3)r^2s^2-2c_3r^4+2c_2s\sqrt{r^2-s^2})}   \\
         &+  \frac{c_0c_1r^6s^2+(c_0+4c_1c_3+2c_1^2)r^4s^2-2c_1c_3r^6+c_1r^4}{r(r^2+(c_1+2c_3)r^2s^2-2c_3r^4+2c_2s\sqrt{r^2-s^2})} .
    \end{split}
\end{equation}
It should be noted that the above two formulas are obtained   in \cite[Eq. (4.9)]{Zhou_Mo} by a long and  completely different way.

 Now, $\phi$ must satisfy    the  condition $\phi_{rs}=\phi_{sr}$ so we have to find what conditions on the functions  $c_0$, $c_1$, $c_2$ and $c_3$ so that     $\phi_{rs}=\phi_{sr}$. For this purpose we have
$$\left(\frac{\phi_s}{\phi}\right)_r-\left(\frac{\phi_r}{\phi}\right)_s=\frac{\phi( \phi_{sr}-\phi_{rs})}{\phi^2}.$$
 That is, $\phi_{rs}=\phi_{sr}$ if and only if $\left(\frac{\phi_s}{\phi}\right)_r=\left(\frac{\phi_r}{\phi}\right)_s$. Applying the condition $\left(\frac{\phi_s}{\phi}\right)_r=\left(\frac{\phi_r}{\phi}\right)_s$ on \eqref{Zhou_phi}, we get the following 
  \begin{align*}
 & \frac{r^2}{\sqrt{r^2-s^2}(r^2+(c_1+2c_3)r^2s^2-2c_3r^4+2c_2s\sqrt{r^2-s^2})^2}
 \Big{(} s\sqrt{r^2-s^2}(  c_1'r^2+2c_3'r^2+8c_3^2r^3\\
 &-2c_0r^3-2c_1^2r^3+4c_0c_1c_3r^7+4c_0c_3r^5+4c_1^2c_3r^5+8c_1c_3^2r^5-2c_0c_1r^5+2c_1c_3'r^4\\
 &-2c_1'c_3r^4)+( 4c_0c_2c_3r^7+4c_1c_2c_3r^5-8c_2c_3^2s^2r^3+2c_0c_2s^2r^3+2c_1c_2s^2r-4c_2c_3's^2r^2\\
 &-4c_2c_3s^2r+4c_2'c_3s^2r^2 +2c_2'r^2-2c_2's^2-4c_0c_2c_3s^2r^5-4c_1c_2c_3s^2r^3+8c_2c_3^2r^5-2c_0c_2r^5\\
 &-2c_1c_2r^3+4c_2c_3'r^3-4c_2'c_3r^4+4c_2c_3r^3 )\Big{)}=0,
\end{align*}
where  $c_1'$ (resp. $c_2'$, $c_3'$ ) is the  derivative  of $c_1$ (resp. $c_2$, $c_3$) with respect to $r$. 
Straightforward simplifications, we have 
\begin{align}
  \label{Eq:Lands_com_1} 
 s \sqrt{r^2-s^2} A(r)+ (r^2-s^2) B(r)=0,
\end{align}
where
 \begin{align*}
  A(r):=& r^2(4c_0c_1c_3r^5+ 2(2c_0c_3+2c_1^2c_3+4c_1c_3^2-c_0c_1)r^3+2(c_1c_3'-c_1'c_3)r^2\\
  &+2(4c_3^2 -c_0-c_1^2)r+2c_3'+c_1')\\
  B(r):=&4c_0c_2c_3r^5+2c_2(2c_1c_3+4c_3^2-c_0)r^3+4(c_2c_3'-c_2'c_3)r^2+2c_2(2c_3-c_1)r+2c_2'.
\end{align*}
 Since the   equation \eqref{Eq:Lands_com_1} is satisfied for all  $s$, then we must have
$$A(r)=0, \quad  B(r)=0.$$
Without loss of generality, we can assume that $c_1$ and $c_3$ are arbitrary functions in $r$ and solve $A(r)=0$ algebraically for $c_0$ we have
\begin{equation}
\label{Eq:c_0}
c_0=-\frac{4c_1c_3(2c_3+c_1)r^3-2(c_1'c_3-c_1c_3')r^2+2(4c_3^2-c_1^2)r+c_1'+2c_3'}{2r(c_1r^2+1)(2c_3r^2-1)}.
\end{equation}
By substituting  $c_0$ into $B(r)=0$ and solving it for $c_2$ we obtain that
\begin{equation}
\label{Eq:c_2}
c_2=c \sqrt{(c_1r^2+1)(2c_3r^2-1)},
\end{equation}
where $c$ is an arbitrary real constant. So if $c_1=-\frac{1}{r^2}$ or $c_3=\frac{1}{2r^2}$, then $c_2=0$ and this implies that the spray coefficients are quadratic and hence the space is Berwaldian. Moreover,  according to Proposition \ref{Regular_condition}, we have to exclude the case where  $\phi$ satisfies  the equation    
  $$ \phi-s\phi_s+(r^2-s^2)\phi_{ss}=0.$$
 Moreover, dividing the above equation by $\phi$, we have
\begin{equation}
\label{Q_defined}
1- \frac{s\phi_s}{\phi}+(r^2-s^2)\frac{\phi_{ss}}{\phi}=0.
\end{equation} 
 By using the property that $\left( \frac{\phi_s}{\phi} \right)_s=\frac{\phi_{ss}}{\phi}-\left( \frac{\phi_s}{\phi} \right)^2 $ and substituting from \eqref{Zhou_phi} into \eqref{Q_defined}, we get that
\begin{equation}
\label{Regularity_2}
-\frac{r^4(c_1r^2+1)(2c_3r^2-1)}{(r^2+(c_1+2c_3)r^2s^2-2c_3r^4+2c_2s\sqrt{r^2-s^2})^2}=0.
\end{equation}  
That is, the choices $c_1=-\frac{1}{r^2}$ or $c_3=\frac{1}{2r^2}$ must be excluded since it implies the non-existence of the function $Q$. Also, the case where $c_1+2c_3=0$ implies that $c_2=c\sqrt{-(2c_3r^2-1)^2}$ which is a contradiction because $c_2$ is a real-valued  function of $r$.  That is, we have the conditions
  \begin{equation}
\label{Eq:c_1_c3}
c_1\neq -\frac{1}{r^2}, \quad c_3\neq \frac{1}{2r^2}, \quad c_1+2c_3\neq 0.
\end{equation}
Taking the equations \eqref{Eq:c_0}, \eqref{Eq:c_2} and \eqref{Eq:c_1_c3} into account, one can find explicit examples of  non-Berwaldian Landsberg metrics of dimension $n\geq 3$. This can be seen by the Examples \ref{Ex:1} and \ref{Ex:2}. 
\end{proof}

\begin{example}\label{Ex:1}
To find an explicit example of a non-Berwaldian Landsberg spherically symmetric metric, consider the functions $c_1$ and $c_3$ are given by
$$c_1=\frac{1}{r^2}, \quad c_3=\frac{1}{r^2}.$$
Then the functions $c_0$  and $c_2$ are calculated by the help of \eqref{Eq:c_0} and \eqref{Eq:c_2} as follows
$$c_0=-\frac{ 3}{ r^4}, \quad c_2=\sqrt{2} c=\frac{1}{2},$$
where we choose the constant $c=\frac{1}{2\sqrt{2}}$.
By substituting the functions $c_0$, $c_1$, $c_2$ and $c_3$ into \eqref{Zhou_phi}, we have
\begin{equation*}
            \frac{\phi_s}{\phi}=  \frac{ 3s+\sqrt{r^2-s^2}}{-r^2+3s^2+ s\sqrt{r^2-s^2}},  \quad \,\, 
         \frac{\phi_r}{\phi}= -\frac{ r}{-r^2+3s^2+ s\sqrt{r^2-s^2}} .
    \end{equation*}
Solving the above equations for $\phi$, we obtain that
\small{\begin{equation}\label{phi_1}
\begin{split}
&\phi(r,s)=(5s^2-r^2)^{\frac{1}{3}}(2s^2-r^2)^{\frac{1}{6}}\exp\Bigg{(} \frac{1}{3}\operatorname{arctanh} \Bigg{(} \frac{\sqrt{5}}{10 }\frac{(\sqrt{5}s+5r)}{\sqrt{r^2-s^2}}\Bigg{)}\\ 
&-\frac{1}{6}\operatorname{arctanh}\Bigg{(}  \frac{\sqrt{2}}{2  }\frac{(\sqrt{2}s+2r)}{\sqrt{r^2-s^2}}\Bigg{)}
+ \frac{1}{3}\operatorname{arctanh} \Bigg{(} \frac{\sqrt{5}}{2  }\frac{(\sqrt{5}s-5r)}{\sqrt{r^2-s^2}}\Bigg{)}-\frac{1}{6}\operatorname{arctanh}\Bigg{(}  \frac{\sqrt{2}}{  2}\frac{(\sqrt{2}s-2r)}{\sqrt{r^2-s^2}}\Bigg{)}\Bigg{)}
.
\end{split}
\end{equation}}
Therefore, the metric $F=u \phi(r,s)$ is a non-Berwladian Landsberg metric with the geodesic spray obtained by the the functions $P$ and $Q$ \eqref{P,Q} and  $\phi$ is given by \eqref{phi_1}. 
\end{example}
Another example can be obtained as the following.
\begin{example}\label{Ex:2}
Let the functions $c_1$ and $c_3$ be given by
$$c_1=0, \quad c_3=\frac{1}{r^2}.$$
Then the functions $c_0$  and $c_2$ are calculated the help of \eqref{Eq:c_0} and \eqref{Eq:c_2} as follows
$$c_0=-\frac{ 2}{ r^4}, \quad c_2=  c=\frac{1}{2},$$
where we choose the constant $c=\frac{1}{2 }$.
By substituting the functions $c_0$, $c_1$, $c_2$ and $c_3$ into \eqref{Zhou_phi}, we have
\begin{equation}
\label{Exmp_2}
            \frac{\phi_s}{\phi}=  \frac{ 2s+\sqrt{r^2-s^2}}{-r^2+2s^2+ s\sqrt{r^2-s^2}},  \quad \,\, 
         \frac{\phi_r}{\phi}= \frac{ -2s^2-s\sqrt{r^2-s^2}}{r(-r^2+2s^2+ s\sqrt{r^2-s^2})} .
    \end{equation}
Solving the first  equation  for $\phi$ by integrating with respect to $s$, we have
\small{$$\phi(r,s)=a(r) (r^4-5r^2s^2+5s^4)^{\frac{1}{4}} \Big{(} \frac{A}{B}\Big{)}^{\frac{5-\sqrt{5}}{20}}\Big{(} \frac{C}{D}\Big{)}^{\frac{5+\sqrt{5}}{20}}\exp\Bigg{(} -\frac{\sqrt{5}}{10}\operatorname{arctanh} \Bigg{(}  \frac{\sqrt{5}(r^2-2s^2)}{r^2}\Bigg{)}\Bigg{)},$$}
where 
\begin{align*}
A:&=\left(2\sqrt{r^2-s^2}-  2r+\left(\sqrt{5}   -1\right)s\right)\sqrt{r^2-s^2}-\left(\sqrt{5} -1 \right)rs,\\
B:&= \left(2\sqrt{r^2-s^2}-2r-\left(\sqrt{5}   -1\right)s\right)\sqrt{r^2-s^2}+\left(\sqrt{5} -1  \right)rs,\\
C:&=\left(2\sqrt{r^2-s^2}-2r-\left(\sqrt{5}   +1\right)s\right)\sqrt{r^2-s^2}+\left(\sqrt{5} +1  \right)rs,\\
D:&=\left(2\sqrt{r^2-s^2}-2r+\left(\sqrt{5}   +1\right)s\right)\sqrt{r^2-s^2}-\left(\sqrt{5} +1  \right)rs.\\
\end{align*}
The function $a(r)$ that satisfies the second equation of \eqref{Exmp_2} is $a(r)=\frac{1}{r}$. That is, the function $\phi$ is given by
 \small{\begin{equation}\label{phi_2}
\phi(r,s)=\frac{1}{r} (r^4-5r^2s^2+5s^4)^{\frac{1}{4}} \Big{(} \frac{A}{B}\Big{)}^{\frac{5-\sqrt{5}}{20}}\Big{(} \frac{C}{D}\Big{)}^{\frac{5+\sqrt{5}}{20}}\exp\Bigg{(} -\frac{\sqrt{5}}{10}\operatorname{arctanh} \Bigg{(}  \frac{\sqrt{5}(r^2-2s^2)}{r^2}\Bigg{)}\Bigg{)}.
\end{equation}}
Therefore, the metric $F=u \phi(r,s)$ is a non-Berwladian Landsberg metric with the geodesic spray obtained by the   functions $P$ and $Q$  and  $\phi$ is given by \eqref{phi_2}.
\end{example}

We end this section by the following two theorems.
\begin{theorem}\label{Th_Riemannian}
All   Berwaldian spherically symmetric metrics of dimension $n\geq3$ are Riemannian. 
\end{theorem}
\begin{proof}
Let $F$ be a Berwald spherically symmetric of  dimension $n\geq3$.  Since every Berwald metric is Landsbergian, then the geodesic spray  of $F$ is given by the functions $P$ and $Q$ \eqref{Zhou_P&Q};
$$P=c_1 s+\frac{c_2}{r^2}\sqrt{r^2-s^2}  , \quad Q=\frac{1}{2}c_0 s^2-\frac{c_2 s}{r^4}\sqrt{r^2-s^2}+c_3.$$
 Moreover, since $F$ is Berwaldian then the mean curvature $E_{ij}$ \eqref{Mean_curv.} vanishes. Substituting by 
 $$P-sP_s=\frac{c_2}{\sqrt{r^2-s^2}}, \quad Q_s-sQ_{ss}=-\frac{c_2}{(r^2-s^2)^{3/2}},$$
  $$ P_{ss}=-\frac{c_2}{(r^2-s^2)^{3/2}}, \quad  Q_{sss}=\frac{3c_2}{(r^2-s^2)^{5/2}}$$
   into the equation $E_{ij}=0$ implies
$$\frac{n c_2}{u\sqrt{r^2-s^2}}\left(  \delta_{ij}-\frac{r^2}{u^2(r^2-s^2)} \delta_{ih} \delta_{jk}y^iy^j -\frac{1}{r^2-s^2} \delta_{ih} \delta_{jk}x^ix^j +\frac{s}{u(r^2-s^2)}(\delta_{ih} \delta_{jk}+\delta_{ik} \delta_{jh})x^hy^k\right)=0.$$
Contracting the above equation by $\delta^{ij}$ and using the facts that $\delta_{ij}y^iy^j=u^2$, $\delta_{ij}x^ix^j=r^2$ and $\delta_{ij}x^iy^j=\langle x,y\rangle$, we have
$$ n (n-2)c_2 \sqrt{r^2-s^2}=0. $$
From which together with the facts that $n\geq3$ and the equation holds for all $r$ and $s$, we must have $c_2=0$. That is, we have
$$P=c_1 s  , \quad Q=\frac{1}{2}c_0 s^2+c_3.$$
Now, the equation \eqref{Zhou_phi} becomes
\begin{equation}
\label{Zhou_phi_2}
    \begin{split}
        \frac{\phi_s}{\phi}= & \frac{(c_1+2c_3)s }{1+(c_1+2c_3)s^2-2c_3r^2 },   \\
         \frac{\phi_r}{\phi}=&   \frac{c_0c_1r^6s^2+(c_0+4c_1c_3+2c_1^2)r^4s^2-2c_1c_3r^6+c_1r^4}{r(r^2+(c_1+2c_3)r^2s^2-2c_3r^4)} ,
    \end{split}
\end{equation}
where $c_0$ is given by \eqref{Eq:c_0}. Integrating $ \frac{\phi_s}{\phi}$ with respect to $s$ implies 
$$\phi=a(r) \sqrt{(c_1+2c_3) s^2-2c_3 r^2+1},$$
where $a(r)$ is to be chosen such that both formulae of  \eqref{Zhou_phi_2} are satisfied, that is, calculating $\frac{\phi_r}{\phi}$ and equaling it with the second formula of \eqref{Zhou_phi_2} we obtain the formula of $a(r)$ as follows
$$a(r)=\exp{\int \frac{r(2c_1c_3 r^2-c_3'r-c_1-2c_3)}{2c_3 r^2 -1} dr}.$$ 
 Consequently, the metric $F=u\phi= a(r) u  \sqrt{(c_1+2c_3) s^2-2c_3 r^2+1} $ is Riemannian.
\end{proof}

\begin{theorem}\label{Reg_Lands_Riem_D_3}
All regular Landsberg spherically symmetric metrics of dimension $n\geq 3$ are Riemannian.
\end{theorem}
\begin{proof}
Let $F=u\phi$ be a regular Landsberg spherically symmetric metric of dimension $n\geq 3$. Now, by the regularity condition \eqref{Regular_condition} together with the fact that $\phi$ is positive, then  we can write
$$1-s\frac{\phi_s}{\phi}+(r^2-s^2)\frac{\phi_{ss}}{\phi}>0.$$
Moreover, by \eqref{Regularity_2} we get
$$1-s\frac{\phi_s}{\phi}+(r^2-s^2)\frac{\phi_{ss}}{\phi}=-\frac{r^4(c_1r^2+1)(2c_3r^2-1)}{(r^2+(c_1+2c_3)r^2s^2-2c_3r^4+2c_2s\sqrt{r^2-s^2})^2}.$$
By \eqref{Eq:c_2}, we must have  $(c_1r^2+1)(2c_3r^2-1)>0$ and hence we have
$$\phi-s\phi_s+(r^2-s^2)\phi_{ss}<0$$
which is a contradiction and hence the metric is Berwaldian. Thus the result follows by Theorem \ref{Th_Riemannian}.
\end{proof}


  \section{The two-dimensional case} 
  
  This section is devoted to the two-dimensional spherically symmetric Finsler metrics. We characterize all spherically symmetric surfaces of Berwald  or Landsberg  types and then we determine  all      Berwald surfaces. We end this section by a note on a class of surfaces mentioned in \cite{Zhou}.
  
  We start by the following proposition. 
  
 \begin{proposition}\label{2D_E_E_ij=0}
 For a spherically symmetric surface the mean Berwald curvature $E_{ij}$ and the scalar trace $E$ of the mean Berwald curvature  vanish if and only if $sH-(r^2-s^2)H_s=0$.
 \end{proposition}
 \begin{proof}
 Assume that $E_{ij}$ vanishes then  contracting \eqref{E_H} by $\delta^{ij}$ together with the fact that $n=2$ imply
 $$ sH-(r^2-s^2)H_s=0.$$
 Also, let the scalar $E$ vanish then by    Proposition \ref{Scalar_E}, we have
$$(\rho_0+\rho_3(r^2-s^2))(sH-(r^2-s^2)H_s)=0.$$
By substituting by  the formulae of $\rho_0$ and $\rho_3$ given in \eqref{Eq:g^ij}, we obtain that
$$\frac{sH-(r^2-s^2)H_s}{\phi(\phi-s\phi_s+(r^2-s^2) \phi_{ss})}=0.$$
Therefore,      $sH-(r^2-s^2)H_s=0$.

Conversely, assume that $ sH-(r^2-s^2)H_s=0$, then by Proposition \ref{Scalar_E} it is clear that $E=0$. Now, taking the fact that $r^2-s^2=\frac{(x_1y_2-x_2y_1)^2}{u^2}$ into account, we can conclude that the components $E_{11}$, $E_{12}$ and $E_{22}$ vanish. For example, we have
$$E_{11}=\frac{y_2^2}{su^3}(s u^2 H-(x_1y_2-x_2y_1)^2 H_s)=\frac{y_2^2}{su }(s  H-(r^2-s^2) H_s)=0.$$ 
Similarly, we can calculate $E_{12}$ and $E_{22}$. Consequently, the proof is completed.
 \end{proof}

 \begin{proposition}\label{2D_L_ijk=0}
A spherically symmetric surface is Landsbergian  if and only if $ (r^2-s^2)L_1+3L_2=0$.
 \end{proposition}
 \begin{proof}
 Assume that the surface is Landsbergian, then     $L_{ijk}=0$. Contracting \eqref{L_ijk} by $x^j\delta^{k\ell}$, we have
  $$(r^2-s^2) L_1+3L_2=0.$$

Conversely, assume that $(r^2-s^2) L_1+3L_2=0$,  one can see that the components $L_{ijk}=0$. For example, we have
$$L_{111}=\frac{\phi}{2u^3}(-x_1^3u^3+3 x_1^2 y_1 s u^2-3x_1 y_1^2 s^2 u+s^3y_1^3)L_1+3(y_1 su^2-x_1 u^3-sy_1^3+x_1y_1^2 u)L_2.$$
Using the formulae of $u$, $r$ and $s$ when $n=2$, one can see that 
$$L_{111}=\frac{y_2^3\phi}{2u^4}(x_1y_2 -x_2 y_1)((r^2-s^2)L_1+3 L_2)=0.$$
Similarly, we  can calculate $L_{112}$, $L_{122}$ and $L_{222}$.  As required.
 \end{proof}
  
  Rewriting  the Landsberg condition for spherically symmetric surfaces and  straightforward calculations imply that  
 \begin{equation}
\label{Main_Lands_surface_cond}
\begin{split}
 (r^2-s^2)L_1+3L_2 =& \lambda_1  \phi_s +\lambda_2 \phi, \\
 \end{split}
\end{equation} where
$$\lambda_1:= \frac{1}{s}(sH-(r^2-s^2)H_s), \quad \lambda_2:=(r^2-s^2)K_s-3sK,\quad K:=P_{ss}-Q_s+sQ_{ss}.$$
  Now, we announce our   first result in this section.     
  \begin{theorem} 
  \label{First_surface_result}
A spherically symmetric Finsler surface is Berwaldian if and only if 
$$P=b_1 s+ \frac{b_2}{ \sqrt{r^2-s^2}} +\frac{b_3 (r^2-2s^2)}{\sqrt{r^2-s^2}} ,$$
$$ Q= b_0 s^2+\frac{1}{2} b_1+ \frac{b_2 s (r^2-2 s^2)}{r^4 \sqrt{r^2-s^2}} -\frac{b_3 s (3r^2-2s^2)}{r^2\sqrt{r^2-s^2}}-\frac{a  }{r^2} s \sqrt{r^2-s^2},$$
where $a$, $b_0$, $b_1$, $b_2$, $b_3$ are arbitrary functions of $r$ and to be chosen such  that the  compatibility conditions are satisfied.
\end{theorem}

\begin{proof}
Let $F$ be a Berwald spherically symmetric Finsler surface, then the mean   curvature $E_{ij}$ and the scalar trace $E$   vanish. Then, by Proposition \ref{2D_E_E_ij=0}, we have  $sH-(r^2-s^2)H_s=0$ which has the solution
\begin{equation}
\label{Eq:Berwald_H_1}
H=\frac{a_0(r)}{\sqrt{r^2-s^2}}.
\end{equation}
 Moreover, since every Berwald metric is Landsbergian, then $F$ is Landsbergian. Now,   by Proposition   \ref{2D_L_ijk=0}, we have
$(r^2-s^2)L_1+3L_2=0$ and since $sH-(r^2-s^2)H_s=0$, then   \eqref{Main_Lands_surface_cond} implies
   $$\lambda_2=(r^2-s^2)K_s-3sK=0.$$
  That is, we have 
  $$\frac{K_s}{K}=\frac{P_{sss}+sQ_{sss}}{P_{ss}-Q_s+sQ_{ss}}=\frac{3s}{r^2-s^2}.$$ 
  Then,  we have the solution
  \begin{equation}
\label{Eq:Berwald_H_2}
K=P_{ss}-Q_s+sQ_{ss}=\frac{b_0(r)}{(r^2-s^2)^{3/2}}.
\end{equation}
One can see that $P_{ss}-Q_s+sQ_{ss}=P_{ss}+(sQ_s)_s-2Q_{s}$, then  integrating the above equation with respect to $s$, we have
 \begin{equation}
\label{Eq:Berwald_H_3}
P_s+sQ_s-2Q=\frac{b_0(r) s}{r^2\sqrt{r^2-s^2}}+b(r).
\end{equation}
Now, by using the definition of $H$ together with  the equations \eqref{Eq:Berwald_H_1} and \eqref{Eq:Berwald_H_2}, we can write
$$(r^2-s^2)P_{ss}-3s P_s+3P=\frac{a_1(r)}{\sqrt{r^2-s^2}}.$$
The   solution of the above equation can be written in the form
\begin{equation}
\label{P_berwald_surface}
P=b_1 s+ \frac{b_2}{ \sqrt{r^2-s^2}} +\frac{b_3 (r^2-2s^2)}{\sqrt{r^2-s^2}} ,
\end{equation}
where $b_1$, $b_2$, $b_3$ are arbitrary functions of $r$. To find the function $Q$, we rewrite   \eqref{Eq:Berwald_H_3} as follows
$$s^3\left( \frac{Q}{s^2}\right)_s= -P_s+\frac{b_0(r) s}{r^2\sqrt{r^2-s^2}}.$$
Hence, by integration  with respect to $s$ we can write the function $Q$ as follows
\begin{equation}
\label{Q_berwald_surface}
Q=b_0 s^2+\frac{1}{2} b_1+ \frac{b_2 s (r^2-2 s^2)}{r^4 \sqrt{r^2-s^2}} -\frac{b_3 s (3r^2-2s^2)}{r^2\sqrt{r^2-s^2}}-\frac{a  }{r^2} s \sqrt{r^2-s^2},
\end{equation}
where $b_0$, $a$ are arbitrary functions of $r$.
 
 Conversely, if $P$ and $Q$ are given by the formulae \eqref{P_berwald_surface} and \eqref{Q_berwald_surface}, then straightforward calculations imply that the Berwald curvature vanishes. Or instead, one can see that  $(r^2-s^2)L_1+3L_2=0$ and $sH-(r^2-s^2)H_s=0$ which    means  that the surface is Landsberian and weekly Berwald ($E_{ij}=0$). But it is known that a Landsberian weekly Berwald surface is Berwaldian and hence the proof is completed.
\end{proof}

     It should be noted that, in \cite{Zhou},   L. Zhou   provided a class of Landsberg spherically symmetric surfaces and he claimed that this class is  not Berwaldian. In fact, this class is Berwaldian as it is shown in  \cite{Elgendi-Youssef}. Also, this can be seen easily since the corresponding  functions $P$ and $Q$   are a special case of  Theorem \ref{First_surface_result}. Moreover, we are going to discuss  another point of Zhou's class related to the compatibility conditions.

 In what follows, to make the comparison easier and  to avoid confusions, we use the terminology and notations of \cite{Zhou}. In the class of metrics  mentioned in  \cite[Theorem 3.3]{Zhou} the  functions $\phi$,  $P$ and $Q$ are not compatible, i.e., the compatibility conditions are not satisfied. We will provide an example showing this non-compatibility soon. The functions $P$ and $Q$ of the class provided in   \cite[Theorem 3.3]{Zhou} are given by   
  \begin{equation}
  \label{P_Q_2D}
  P=-\frac{s}{r^2}+\frac{c}{r^2}\sqrt{r^2-s^2}, \quad Q=c_0-\frac{4r^4c_0^2+2r^3c_0'+c^2}{2r^4(2r^2c_0-1)}s^2-\frac{s}{r^4}\sqrt{r^2-s^2},
  \end{equation}
  where $c_0$  is function of  $r$ and   $c\neq 3$ is a constant, $c_0'$ denotes the derivative of $c_0$ with respect to $r$. By the formula of $Q$, it is clear that the case where $c_0=\frac{1}{2r^2}$ is excluded.

  Plugging the functions $P$ and $Q$ given by \eqref{P_Q_2D} into the compatibility conditions \eqref{Comp_C_C_2},  then  straightforward calculations by hand or by using Maple program,  we have the following two equations:
$$\frac{\phi_s}{\phi}(r^2-s^2)( (2c_0r^2-1)\sqrt{r^2-s^2} -(c+1)s)+  (2c_0r^2-1)s\sqrt{r^2-s^2}+2cr^2-(c+1)s^2=0,$$
\begin{equation*}
    \begin{split}
         &\frac{\phi_r}{\phi} (2c_0r^2-1)r^2\sqrt{r^2-s^2}+\frac{\phi_s}{\phi} \Big{(} \sqrt{r^2-s^2}( (r^2-s^2)(4c_0^2sr^4+sc^2)+(2c_0r^2-1)sr^2) \\
         &+  3s^2r^2+2c_0cs^2r^4+2c_0r^6+cr^4-2c_0cr^6-cr^2s^2-6c_0s^2r^4+4c_0s^4r^2-r^4-2s^4\Big{)} \\
         &+\sqrt{r^2-s^2}(4c_0^2s^2r^4+2c_0r^4+c^2s^2-r^2)+2c_0csr^4-sr^2+2c_0sr^4-csr^2-4c_0s^3r^2+2s^3=0.
    \end{split}
\end{equation*}

By solving the above equations  for $\frac{\phi_s}{\phi}$ and $\frac{\phi_r}{\phi}$, we have
  
 \small{ \begin{equation}
\label{Zhou)_2D_phi}
    \begin{split}
        \frac{\phi_s}{\phi}= &- \frac{(c+1)s^2-(2r^2c_0-1)s\sqrt{r^2-s^2}-2r^2c}{(r^2-s^2)((2r^2c_0-1)\sqrt{r^2-s^2}-(c+1)s)},   \\
         \frac{\phi_r}{\phi}=& - \frac{\sqrt{r^2-s^2}(2c(2c_0r^2-1)(1-c)(r^2-s^2)-r^2+2s^2+4c_0r^4+8c_0^2r^4s^2-4c_0^2r^6-8c_0r^2s^2)}{r(r^2-s^2)(2r^2c_0-1)((c+1)s-(2r^2c_0-1)\sqrt{r^2-s^2})}   \\
         &+ \frac{5cr^2s-2c^3r^2s-4cc_0'r^5s+12cc_0r^2s^3-14cc_0r^4s+4cc_0'r^3s^3+r^2s-4cs^3-2c_0r^4s}{r(r^2-s^2)(2r^2c_0-1)((c+1)s-(2r^2c_0-1)\sqrt{r^2-s^2})}\\
        &-\frac{4c_0r^2s^3+2c^3s^3-2s^3}{r(r^2-s^2)(2r^2c_0-1)((c+1)s-(2r^2c_0-1)\sqrt{r^2-s^2})} .
    \end{split}
\end{equation}}

     Now,  consider  the following example.

\begin{example}  This example is constructed from \cite[Theorem 3.3]{Zhou} and shows a non-compatibility between the functions $P$, $Q$ and $\phi$. Let $c=-1$ and $c_0=\frac{1}{r^2}$, then, by  substitution in the  formula of $ a(r)$, we have
  $$a(r)=\exp\left( \int -\frac{2c_0r^2-1+2c^2-2c}{r(2c_0r^2-1)}dr\right)=\exp\left( \int -\frac{5}{r}dr\right)=\frac{1}{r^5}.$$
  Now, we calculate   $F$  as follows
\begin{align*}
  F & =u\exp \left(\int  \frac{(c+1) s^{2}-\left(2 r^{2} c_{0}-1\right) s \sqrt{r^{2}-s^{2}}-2 r^{2} c}{\left(r^{2}-s^{2}\right)\left(\left(2 c_{0} r^{2}-1\right) \sqrt{r^{2}-s^{2}}-(c+1) s\right)} d s\right) a(r) \\
    & = \frac{u}{r^5}\exp \left(\int \frac{ -  s \sqrt{r^{2}-s^{2}}+2 r^{2}  }{\left(r^{2}-s^{2}\right)\left(\sqrt{r^{2}-s^{2}}\right)} d s\right)\\
    &=\frac{u}{r^5}  \sqrt{r^{2}-s^{2}}\exp \left(  \frac{2 s  }{ \sqrt{r^{2}-s^{2}} } \right).
\end{align*}
Using \eqref{P,Q} and the above formula of  $\phi(r,s)$,  we obtain that
\begin{equation}
\label{P_Q_from_phi}
P= -\frac{s}{r^2}-\frac{3}{4 r^2}\,\sqrt {{r}^{2}-{s}^{2}}, \quad Q=\frac{7}{8 r^2}-\frac{3 s^2}{r^4}-\frac{3s}{4 r^4}\,\sqrt {{r}^{2}-{s}^{2}}
\end{equation}
 On the other hand,  by \eqref{P_Q_2D}, the choice $c=-1$ and $c_0=\frac{1}{r^2}$ leads to 
  \begin{equation}
  \label{P_Q_from_c's}
P= -\frac{s}{r^2}-\frac{1}{r^2}\,\sqrt {{r}^{2}-{s}^{2}}, \quad Q=\frac{1}{r^2}-\frac{s^2}{2r^4}-\frac{s}{r^4}\,\sqrt {{r}^{2}-{s}^{2}}
.
\end{equation}
  It is clear that the functions $P$ and $Q$ in the equations \eqref{P_Q_from_phi} and \eqref{P_Q_from_c's} are different. This is because the functions $P$,  $Q$ and $\phi$ do not satisfy the compatibility conditions, one can see that the compatibility conditions are $C_1=0$ and $C_2=\frac{1}{r^2} \phi$.
\end{example}

The problem of the non-compatibility in the previous example comes from the factor $a(r)$. In \cite{Zhou}, it is mentioned that 
the function $a(r)$ is given by  $\phi(r,0)=a(r)$ and it is not clear   how   this formula is obtained, in the above example, one can see that $\phi(r,0)=\frac{1}{r^4}\neq a(r)$.
\medskip

In the previous example, by substitution by $c=-1$ and $c_0=\frac{1}{r^2}$ into  \eqref{Zhou)_2D_phi}, we have
$$\frac{\phi_s}{\phi}=-\frac{r^2s-s^3-2r^2\sqrt{r^2-s^2} }{(r^2-s^2)^2}, \quad \frac{\phi_r}{\phi}=-\frac{5r^4-11r^2s^2+6s^4-2r^2s\sqrt{r^2-s^2} }{r(r^2-s^2)^2}.$$
  Solving the above equations, we get that
  $$\phi=\frac{1}{r^6}\sqrt{r^{2}-s^{2}}\exp \left(  \frac{2 s  }{ \sqrt{r^{2}-s^{2}} } \right).$$
  That is, $a(r)=\frac{1}{r^6}$. Moreover,  one can see that   the above formula of $\phi$ together with \eqref{P,Q} yield the same formulae \eqref{P_Q_from_c's} of $P$ and $Q$.

\end{document}